\documentclass[a4paper,11pt]{article}
\usepackage{amssymb}
\usepackage{amsmath}
\usepackage{amsfonts}

\providecommand{\U}[1]{\protect \rule{.1in}{.1in}}
\newtheorem{theorem}{Theorem}

\newtheorem{corollary}[theorem]{Corollary}

\newtheorem{example}{Example}

\newtheorem{proposition}[theorem]{Proposition}
\newtheorem{remark}[theorem]{Remark}

\newenvironment{proof}[1][Proof]{\noindent \textbf{#1.} }{\  \rule[]{0.5em}{0.5em}}
\include {mak}
\input{tcilatex}
\begin{document}

\begin{center}
{\LARGE Note on the real rootedness of polynomials}

\  \ 

{\Large Abdelkader Benyattou and Miloud Mihoubi} \  \\[0pt]
\  \  \\[0pt]
USTHB, Faculty of Mathematics, RECITS Laboratory, Algiers, Algeria.

abdelkaderbenyattou@gmail.com \  \  \ mmihoubi@usthb.dz

\  \  \  \  \ 
\end{center}

\noindent \textbf{Abstract.} In this paper, by the generalized Bell umbra
and Rolle's theorem, we give some results on the real rootedness of
polynomials. Some applications on partition polynomials and the sigma
polynomials of graphs are given. 



\section{Introduction}

The real rootedness of polynomials has attracted researchers great interest.
One of the reasons is that any polynomial of real zeros implies the
log-concavity and the unimodality of its coefficients, which appear in
various fields of mathematics, see \cite{bre,sta}. In this paper, we
investigate the properties of the generalized Bell umbra and Rolle's theorem
to give a result on the real rootedness of polynomials. Partial and
auxiliary results on partition polynomials and exponential polynomials are
also given. For the partition polynomials, we consider the $\sigma $%
-polynomials of graphs and a class of polynomials linked to the partial $r$%
-Bell polynomials. The mathematical tools used here are the generalized Bell
umbra and Rolle's theorem. To use them, recall that the $n$-th Bell
polynomial $\mathcal{B}_{n}\left( x\right) $ and the $n$-th $r$-Bell
polynomial $\mathcal{B}_{n,r}\left( x\right) $ can be defined by Dobinski's
formula:%
\begin{equation*}
\mathcal{B}_{n}\left( x\right) =e^{-x}\underset{j\geq 0}{\sum }j^{n}\frac{%
x^{j}}{j!},\text{ \ }\mathcal{B}_{n,r}\left( x\right) =e^{-x}\underset{j\geq
0}{\sum }\left( j+r\right) ^{n}\frac{x^{j}}{j!}
\end{equation*}%
and $\mathbf{B}_{\mathbf{x}}^{n}=\mathcal{B}_{n}\left( x\right) $ be the
generalized Bell umbra introduced by Sun et al. \cite{sun}. \newline
For further information about umbral calculus on Bell polynomials, one can
also see \cite{Benyattou,ger,ges}. \newline
To use later, recall that for any polynomial $f$ and integer $n\geq 0,$ it
is known \cite{sun} that 
\begin{equation*}
\left( \mathbf{B}_{\mathbf{x}}\right) _{n}f\left( \mathbf{B}_{\mathbf{x}%
}\right) =x^{n}f\left( \mathbf{B}_{\mathbf{x}}+n\right) .
\end{equation*}%
In particular, we have 
\begin{equation*}
\mathbf{B}_{\mathbf{x}}^{n+1}=x\left( \mathbf{B}_{\mathbf{x}}+1\right) ^{n}%
\text{ \ and \ }\left( \mathbf{B}_{\mathbf{x}}\right) _{n}=x^{n}.
\end{equation*}%
where $\left( x\right) _{n}=x\left( x-1\right) \cdots \left( x-n+1\right) $
if $n\geq 1$ and $\left( x\right) _{0}=1.$\newline
The paper is organized as follows. In the next section we give a result on
polynomials with real zeros. In the third section we present an application
on the $\sigma $-polynomials and another on a class of polynomials linked to
the partial $r$-Bell polynomials \cite{Mihoubi2}. In the last section we
give an application of a class of exponential polynomials.

\section{Polynomials of real zeros via generalized Bell umbra}

Let $\mathbf{RZ}$ be the set of real polynomials having only real zeros. 
\newline
The principal main result of this paper is the following theorem.

\begin{theorem}
\label{T1}Let $r$ be a non-negative integer and let $f$ be a polynomial with
real coefficients such that $f\left( \mathbf{B}_{\mathbf{x}}\right) \in 
\mathbf{RZ}.$ Then, for any non-negative integers $r_{1},\ldots ,r_{q},$
there holds 
\begin{equation*}
\left( \mathbf{B}_{\mathbf{x}}\right) _{r_{p}}\cdots \left( \mathbf{B}_{%
\mathbf{x}}\right) _{r_{1}}f\left( \mathbf{B}_{\mathbf{x}}\right) \in 
\mathbf{RZ}.
\end{equation*}
\end{theorem}

\begin{proof}
From the identity $f\left( \mathbf{B}_{\mathbf{x}}\right) =e^{-x}\sum_{k\geq
0}f\left( k\right) \frac{x^{k}}{k!}$ \cite{Benyattou} we get%
\begin{equation*}
\frac{d}{dx}\left( e^{x}f\left( \mathbf{B}_{\mathbf{x}}+r-1\right) \right) =%
\frac{d}{dx}\left( \sum_{k\geq 0}f\left( k+r-1\right) \frac{x^{k}}{k!}%
\right) =\sum_{k\geq 0}f\left( k+r\right) \frac{x^{k}}{k!}=e^{x}f\left( 
\mathbf{B}_{\mathbf{x}}+r\right).
\end{equation*}%
The proof can be obtained by induction on $r$ by application of Rolle's
theorem on the function $e^{x}f\left( \mathbf{B}_{\mathbf{x}}+r-1\right)$.
More generally, since $g\left( x\right) : =\left( \mathbf{B}_{\mathbf{x}%
}\right) _{r_{1}}f\left( \mathbf{B}_{\mathbf{x}}\right) \in \mathbf{RZ}$, it
follows $\left( \mathbf{B}_{\mathbf{x}}\right) _{r_{2}}g\left( \mathbf{B}_{%
\mathbf{x}}\right)$ $=\left( \mathbf{B}_{\mathbf{x}}\right) _{r_{2}}\left( 
\mathbf{B}_{\mathbf{x}}\right) _{r_{1}}f\left( \mathbf{B}_{\mathbf{x}%
}\right) \in \mathbf{RZ}$ and so on.
\end{proof}

\begin{example}
Let $f\left( x\right) =x^{n},$ then $f\left( \mathbf{B}_{\mathbf{x}}\right) =%
\mathcal{B}_{n}\left( x\right) $ which has only real zeros. We deduce the
known result \cite{Maamra} (see also \cite{Benyattou,Mihoubi0}): 
\begin{equation*}
\left( \mathbf{B}_{\mathbf{x}}\right) _{r_{p}}\cdots \left( \mathbf{B}_{%
\mathbf{x}}\right) _{r_{1}}\mathbf{B}_{\mathbf{x}}^{n}=x^{max\{r_{1},\ldots
,r_{p}\}}\mathcal{B}_{n;r_{1},\ldots,r_{p}}\left( x\right) \in \mathbf{RZ}.
\end{equation*}
\end{example}

\begin{example}
Let $f\left( x\right) =\left( x\right) _{n},$ then $f\left( \mathbf{B}_{%
\mathbf{x}}\right) =x^{n}\in \mathbf{RZ}$. It follows that the polynomial 
\begin{equation*}
\left( \mathbf{B}_{\mathbf{x}}\right) _{r}f\left( \mathbf{B}_{\mathbf{x}%
}\right) =x^{r}\left( \mathbf{B}_{\mathbf{x}}+r\right) _{n}=x^{r}\underset{%
k=0}{\overset{\min \left( n,r\right) }{\sum }}\binom{n}{k}\frac{r!}{\left(
r-k\right) !}x^{n-k}
\end{equation*}%
is in $\mathbf{RZ}$. Also, the polynomial 
\begin{equation*}
\left( \mathbf{B}_{\mathbf{x}}\right) _{r_{p}}\cdots \left( \mathbf{B}_{%
\mathbf{x}}\right) _{r_{1}}\left( \mathbf{B}_{\mathbf{x}}\right)
_{n}=x^{r_{p}}\underset{k=0}{\overset{n}{\sum }}\left( -1\right) ^{n-k}{n %
\brack k}\mathcal{B}_{k;r_{1},\cdots,r_{p}}\left( x\right)
\end{equation*}%
is in $\mathbf{RZ},$ where ${n \brack k}$ are the unsigned Stirling numbers
of the first kind.
\end{example}

\begin{example}
Let $f\left( x\right) =\left( x+n-1\right) _{n},$ then $f\left( \mathbf{B}_{%
\mathbf{x}}\right) =\left( \mathbf{B}_{\mathbf{x}}+n-1\right) _{n}=\mathcal{L%
}_{n}\left( x\right) $ \cite{Benyattou} is the $n$-th Lah polynomial which
is in $\mathbf{RZ}$. It follows that the polynomial 
\begin{equation*}
\left( \mathbf{B}_{\mathbf{x}}\right) _{r_{p}}\cdots \left( \mathbf{B}_{%
\mathbf{x}}\right) _{r_{1}}\left( \mathbf{B}_{\mathbf{x}}+n-1\right)
_{n}=x^{r_{p}}\underset{k=0}{\overset{n}{\sum }}\left( -1\right) ^{n-k}{2n-1 %
\brack k+n-1}_{n-1}\mathcal{B}_{k;r_{1},\cdots,r_{p}}\left( x\right)
\end{equation*}%
is in $\mathbf{RZ},$ where ${n \brack k}_{r}$ are the unsigned $r$-Stirling
numbers of the first kind.
\end{example}

\section{Partition polynomials with real zeros}

We present in this section two applications of Theorem \ref{T1}: one on the $%
\sigma$-polynomial associated to a graph and another on a class of
polynomials linked to the partial $r$-Bell polynomials. \newline
For the first application, recall that a $\lambda$-coloring of $G,$ $\lambda
\in \mathbb{N},$ is a mapping $f:V\rightarrow \left \{ 1,2,\ldots,\lambda
\right \} $ where $f(u)\neq f(v)$ whenever the vertices $u$ \ and $v$ are
adjacent in $G. $ Two $\lambda $-colorings $f$ and $g$ of $G$ are distinct
if $f(x)\neq g(x)$ for some vertex $x$ in $G,$ and, the number of $\lambda$%
-colorings of $G$ is called the chromatic polynomial $P(G,\lambda).$ The
chromatic polynomial can be defined as $f\left( \lambda \right) =\underset{%
k=0}{\overset{n}{\sum}}\alpha _{k}\left( G\right) \left( \lambda \right)
_{k},$ where $\alpha_{k}\left( G\right) $ is the number of ways of
partitioning $V$ into $i$ nonempty sets. The $\sigma$-polynomial associated
to $G$ is $\underset{k=0}{\overset{n}{\sum }}\alpha_{k}\left( G\right) x^{k}=%
\underset{k=0}{\overset{n}{\sum}}\alpha_{k}\left( G\right) \left( \mathbf{B}%
_{\mathbf{x}}\right) _{k}=f\left( \mathbf{B}_{\mathbf{x}}\right).$ \newline
For more information about chromatic polynomials, see \cite{don}.

\begin{corollary}
If the $\sigma$-polynomial $f\left( \mathbf{B}_{\mathbf{x}}\right)$ of a
graph $G$ is in $\mathbf{RZ}$, then the $\sigma $-polynomial%
\begin{equation*}
\left( \mathbf{B}_{\mathbf{x}}\right) _{r_{1}}\cdots \left( \mathbf{B}_{%
\mathbf{x}}\right) _{r_{p}}f\left( \mathbf{B}_{\mathbf{x}}\right)
\end{equation*}%
of the graph $G\cup K_{r_{1}}\cup \cdots \cup K_{r_{p}}$ is in $\mathbf{RZ},$
where $K_{r}$ is the complete graph of $r$ vertices.
\end{corollary}

\begin{example}
A tree $T_{n}$ of $n$ $\left( \geq 1\right) $ vertices has $f\left( x\right)
=x\left( x-1\right) ^{n-1}.$ By the identity $\left( \mathbf{B}_{\mathbf{x}%
}\right) _{n}f\left( \mathbf{B}_{\mathbf{x}}\right)$ $=x^{n}f\left( \mathbf{B%
}_{\mathbf{x}}+n\right),$ the $\sigma $-polynomial is to be $f\left( \mathbf{%
B}_{\mathbf{x}}\right) =\mathbf{B}_{\mathbf{x}}\left( \mathbf{B}_{\mathbf{x}%
}-1\right) ^{n-1}=x\mathbf{B}_{\mathbf{x}}^{n-1}=x\mathcal{B}_{n-1}\left(
x\right)$ which is in $\mathbf{RZ}.$ Then, the $\sigma $-polynomial of the
graph $T_{n}\cup K_{r}$ is in $\mathbf{RZ}$, and is to be 
\begin{equation*}
x^{r}f\left( \mathbf{B}_{\mathbf{x}}+r\right) =x^{r}\left( \mathbf{B}_{%
\mathbf{x}}+r\right) \left( \mathbf{B}_{\mathbf{x}}+r-1\right) ^{n-1}=x^{r}%
\left[ x\mathcal{B}_{n-1,r}\left( x\right) +r\mathcal{B}_{n-1,r-1}\left(
x\right) \right].
\end{equation*}
\end{example}

\noindent For the second application, let the $\left( n,k\right) $-th
partial $r$-Bell polynomial%
\begin{equation*}
B_{n+r,k+r}^{\left( r\right) }\left( \mathbf{a};\mathbf{b}\right)
:=B_{n+r,k+r}^{\left( r\right) }\left( a_{1},a_{2},\ldots;b_{1},b_{2},\ldots
\right)
\end{equation*}
introduced by Mihoubi et al. \cite{Mihoubi2} (see also \cite{sha}) and
defined by%
\begin{equation*}
\underset{n\geq k}{\sum}B_{n+r,k+r}^{\left( r\right) }\left( \mathbf{a};%
\mathbf{b}\right) \frac{t^{n}}{n!}=\frac{1}{k!}\left( \underset{j\geq 1}{\sum%
}a_{j}\frac{t^{j}}{j!}\right) ^{k}\left( \underset{j\geq0}{\sum }b_{j+1}%
\frac{t^{j}}{j!}\right) ^{r}.
\end{equation*}

\noindent This polynomial presents an extension of the $\left( n,k\right)$%
-th partial Bell polynomial $B_{n,k}\left( a_{1},a_{2},\ldots
\right):=B_{n,k}\left(\mathbf{a}\right)$ introduced by Bell \cite{Bell} and
studied later by several authors, see for example \cite%
{Comtet,Mihoubi3,Mihoubi1}). \newline
Let $\left( a_{n}\right) $ and $\left( b_{n}\right) $ be two the sequences
of real numbers linked as follows%
\begin{equation*}
\varphi \left( t\right) =\underset{n\geq1}{\sum}a_{n}\frac{t^{n}}{n!},\  \  \
1+\varphi \left( t\right) =\underset{n\geq0}{\sum}b_{n+1}\frac{t^{n}}{n!},
\end{equation*}
Here $\mathbf{b=e}+L\mathbf{a,}$ where $\mathbf{e=}\left( 1,0,0,\ldots
\right)$, $\mathbf{a=}\left( a_{1},a_{2},\ldots \right),$ and the sequence $%
\left(L^{n}\mathbf{a}\right)$ is defined by $L^{0}\mathbf{a=}\left(
a_{1},a_{2},\ldots \right)$, $L\mathbf{a=}\left( 0,a_{1},a_{2},\ldots
\right),\ L^{2}\mathbf{a=}\left( 0,0,a_{1},a_{2},\ldots \right)$ and so on.

\begin{proposition}
\label{P5}Let $\mathcal{V}_{n,r}\left( x\right) $ and $\mathcal{V}_{n}\left(
x\right) $ be the polynomials defined by%
\begin{equation*}
\mathcal{V}_{n,r}\left( x\right) =\underset{k=0}{\overset{n}{\sum }}%
B_{n+r,k+r}^{\left( r\right) }\left( \mathbf{a};\mathbf{e}+L\mathbf{a}%
\right) x^{k},\text{ \ }\mathcal{V}_{n}\left( x\right) =\mathcal{V}%
_{n,0}\left( x\right) =\underset{k=0}{\overset{n}{\sum }}B_{n,k}\left( 
\mathbf{a}\right) x^{k}.
\end{equation*}%
If $\mathcal{V}_{n}\left( x\right) \in \mathbf{RZ},$ then $\mathcal{V}%
_{n,r}\left( x\right) \in \mathbf{RZ}.$
\end{proposition}

\begin{proof}
Then, from \cite[Th. 4]{Mihoubi2} we have%
\begin{align*}
\underset{n\geq0}{\sum}B_{n+r,k+r}^{\left( r\right) }\left( \mathbf{a};%
\mathbf{e}+L\mathbf{a}\right) \frac{t^{n}}{n!} & =\frac{1}{k!}\left( \varphi
\left( t\right) \right) ^{k}\left( 1+\varphi \left( t\right) \right) ^{r}, \\
\underset{n\geq0}{\sum}\mathcal{V}_{n,r}\left( x\right) \frac{t^{n}}{n!} &
=\left( 1+\varphi \left( t\right) \right) ^{r}\exp \left( x\varphi \left(
t\right) \right).
\end{align*}
For $f_{n}\left( x\right) =\underset{k=0}{\overset{n}{\sum}}B_{n,k}\left( 
\mathbf{a}\right) \left( x\right) _{k}$ we get $f_{n}\left( \mathbf{B}_{%
\mathbf{x}}+r\right) =\mathcal{V}_{n,r}\left( x\right).$ Indeed, we have%
\begin{align*}
\underset{n\geq0}{\sum}f_{n}\left( \mathbf{B}_{\mathbf{x}}+r\right) \frac{%
t^{n}}{n!} & =\underset{k\geq0}{\sum}\left( \mathbf{B}_{\mathbf{x}}+r\right)
_{k}\underset{n\geq k}{\sum}B_{n,k}\left( \mathbf{a}\right) \frac{t^{n}}{n!}
\\
& =\underset{k\geq0}{\sum}\binom{\mathbf{B}_{\mathbf{x}}+r}{k}\left( \varphi
\left( t\right) \right) ^{k} \\
& =\left( 1+\varphi \left( t\right) \right) ^{\mathbf{B}_{\mathbf{x}}+r} \\
& =\left( 1+\varphi \left( t\right) \right) ^{r}\underset{n\geq0}{\sum }%
f_{n}\left( \mathbf{B}_{\mathbf{x}}\right) \frac{t^{n}}{n!} \\
& =\left( 1+\varphi \left( t\right) \right) ^{r}\exp \left( x\varphi \left(
t\right) \right),
\end{align*}
which is the exponential generating function of the sequence $\left (%
\mathcal{V} _{n,r}\left( x\right)\right)$. \newline
Hence, the application of Theorem \ref{T1} completes the proof.
\end{proof}

\begin{corollary}
For $\mathbf{a=}\left( 1,1,1,\ldots \right),$ the following polynomials are
in $\mathbf{RZ}$ 
\begin{equation*}
\mathcal{V}_{n,r}\left( x\right) =\sum_{k=0}^{n}B_{n+r,k+r}^{\left( r\right)
}\left( L\mathbf{a};\mathbf{e}+L^{2}\mathbf{a}\right) x^{k},\  \  \  \mathcal{U}%
_{n,r}\left( x\right) =\sum_{k=0}^{n}B_{n+r,k+r}^{\left( r\right) }\left(
L^{2}\mathbf{a};\mathbf{e}+L^{3}\mathbf{a}\right) x^{k}
\end{equation*}
\end{corollary}

\begin{proof}
The $2$-associated and $3$-associated Bell polynomials%
\begin{equation*}
\mathcal{V}_{n}\left( x\right) =\sum_{k=0}^{n}{n \brace k}^{\left( 2\right)
}x^{k}=\sum_{k=0}^{n}B_{n,k}\left( L\mathbf{a}\right) x^{k},\text{ \ }%
\mathcal{U}_{n}\left( x\right) =\underset{k=0}{\overset{n}{\sum }}{n \brace k%
}^{\left( 3\right) }x^{k}=\sum_{k=0}^{n}B_{n,k}\left( L^{2}\mathbf{a}\right)
x^{k}
\end{equation*}%
are in $\mathbf{RZ},$ where%
\begin{equation*}
\underset{n\geq k}{\sum }{n \brace k}^{\left( 2\right) }\frac{t^{n}}{n!}=%
\frac{1}{k!}\left( e^{t}-1-t\right) ^{k},\  \  \  \underset{n\geq k}{\sum }{n %
\brace k}^{\left( 3\right) }\frac{t^{n}}{n!}=\frac{1}{k!}\left( e^{t}-1-t-%
\frac{t^{2}}{2}\right) ^{k},
\end{equation*}%
see \cite{Bona,Tebtoub}. So, the corollary follows from Proposition \ref{P5}.
\end{proof}

\section{Exponential polynomials with real zeros}

The following proposition presents an application of Theorem 1 given in \cite%
{bender}.

\begin{theorem}
\label{T3}Let $\left( A_{n}\left( x\right) \right) $ be a sequence of
polynomials defined by%
\begin{equation*}
1+\underset{n\geq1}{\sum}A_{n}\left( x\right) \frac{t^{n}}{n!}=\exp \left(
xh\left( t\right) \right),\  \ h\left( t\right) =\underset{j\geq1}{\sum }a_{j}%
\frac{t^{j}}{j!},
\end{equation*}
and let $\left( A_{n}^{\left( s\right) }\left( x\right) \right) $ be a
sequence of polynomials defined by%
\begin{equation*}
A_{n}^{\left( 0\right) }\left( x\right) =A_{n}\left( x\right),\  \
A_{n}^{\left( s\right) }\left( x\right) =A_{n}^{\left( s-1\right) }\left( 
\mathbf{B}_{\mathbf{x}}\right),\text{ \ }s\geq1.
\end{equation*}
Then%
\begin{equation*}
1+\underset{n\geq1}{\sum}A_{n}^{\left( s\right) }\left( x\right) \dfrac{t^{n}%
}{n!}=\exp \left( x\underset{j\geq1}{\sum}A_{j}^{\left( s-1\right) }\left(
1\right) \dfrac{t^{j}}{j!}\right),\text{ \ }s\geq1.
\end{equation*}
Furthermore, if the sequence $\left( \frac{A_{n}^{\left( s\right) }\left(
1\right) }{\left( n-1\right) !}\right) $ is log-concave, then for $x>0,$ the
sequence $\left( A_{n}^{\left( s\right) }\left( x\right) \right) $ is
log-convex and the sequence $\left( \frac{A_{n}^{\left( s\right) }\left(
x\right) }{n!}\right) $ is log-concave.
\end{theorem}

\begin{proof}
The desired identity is true for $s=1$ because we have%
\begin{align*}
\underset{n\geq 0}{\sum }A_{n}^{\left( 1\right) }\left( x\right) \dfrac{t^{n}%
}{n!}& =\underset{n\geq 0}{\sum }A_{n}^{\left( 0\right) }\left( \mathbf{B}_{%
\mathbf{x}}\right) \dfrac{t^{n}}{n!} \\
& =\exp \left( \mathbf{B}_{\mathbf{x}}h\left( t\right) \right)  \\
& =\underset{n\geq 0}{\sum }\dfrac{\mathcal{B}_{n}\left( x\right) }{n!}%
\left( h\left( t\right) \right) ^{n} \\
& =\exp \left( x\left( \exp \left( h\left( t\right) \right) -1\right)
\right)  \\
& =\exp \left( x\underset{n\geq 1}{\sum }A_{n}^{\left( 0\right) }\left(
1\right) \frac{t^{n}}{n!}\right) .
\end{align*}%
Assume it is true for $s,$ $s\geq 1$. Then%
\begin{align*}
\underset{n\geq 0}{\sum }A_{n}^{\left( s+1\right) }\left( x\right) \dfrac{%
t^{n}}{n!}& =\underset{n\geq 0}{\sum }A_{n}^{\left( s\right) }\left( \mathbf{%
B}_{\mathbf{x}}\right) \dfrac{t^{n}}{n!} \\
& =\exp \left( \mathbf{B}_{\mathbf{x}}\underset{j\geq 1}{\sum }A_{j}^{\left(
s-1\right) }\left( 1\right) \dfrac{t^{j}}{j!}\right)  \\
& =\underset{n\geq 0}{\sum }\dfrac{\mathcal{B}_{n}\left( x\right) }{n!}%
\left( \underset{j\geq 1}{\sum }A_{j}^{\left( s-1\right) }\left( 1\right) 
\frac{t^{j}}{j!}\right) ^{n} \\
& =\exp \left( x\left( \exp \left( \underset{j\geq 1}{\sum }A_{j}^{\left(
s-1\right) }\left( 1\right) \frac{t^{j}}{j!}\right) -1\right) \right)  \\
& =\exp \left( x\underset{n\geq 1}{\sum }A_{n}^{\left( s\right) }\left(
x\right) \dfrac{t^{n}}{n!}\right) ,
\end{align*}%
which proves the induction step. To rest of the proof can be obtained by
induction on $s$ upon using Theorem 1 given by Bender-Canfield \cite{bender}.
\end{proof}

\begin{example}
For $a_{n}=1,$ we get Theorem 1 given in \cite[Th. 1]{Asai}.
\end{example}

\begin{example}
For $a_{n}=\left( n-1\right) !,\ n\geq 1,$ we get $h\left( t\right) =-\ln
\left( 1-t\right) $ and%
\begin{align*}
A_{n}^{\left( 0\right) }\left( x\right) & =x\left( x+1\right) \cdots \left(
x+n-1\right) :=\left \langle x\right \rangle _{n}\text{ with }\left \langle
x\right \rangle _{0}=1, \\
A_{n}^{\left( 1\right) }\left( x\right) & =A_{n}\left( \mathbf{B}_{\mathbf{x}%
}\right) =\left( \mathbf{B}_{\mathbf{x}}+n-1\right) _{n}=\mathcal{L}%
_{n}\left( x\right) , \\
A_{n}^{\left( 2\right) }\left( x\right) & =A_{n}^{\left( 2\right) }\left( 
\mathbf{B}_{\mathbf{x}}\right) =\mathcal{L}_{n}\left( \mathbf{B}_{\mathbf{x}%
}\right) =\underset{k=0}{\overset{n}{\sum }}L(n,k)\mathcal{B}_{k}\left(
x\right) ,\text{ etc},
\end{align*}%
where $L(n,k)$ are the Lah numbers. Theorem \ref{T3} shows that the
sequences $\left( \left \langle x\right \rangle _{n}\right) $, $\left( 
\mathcal{L}_{n}\left( x\right) \right) $ and $\left( \underset{k=0}{\overset{%
n}{\sum }}L(n,k)\mathcal{B}_{k}\left( x\right) \right) $ are log-convex and
the sequences $\left( \frac{\left \langle x\right \rangle _{n}}{n!}\right) ,$ $%
\left( \frac{\mathcal{L}_{n}\left( x\right) }{n!}\right) $ and $\left( 
\underset{k=0}{\overset{n}{\sum }}\frac{L(n,k)}{n!}\mathcal{B}_{k}\left(
x\right) \right) $ are log-concave.
\end{example}

\begin{remark}
Let $\mathbf{a}_{s}\mathbf{=}\left( A_{1}^{\left( s\right) }\left( 1\right)
,A_{2}^{\left( s\right) }\left( 1\right) ,\ldots \right) $ and let the
polynomials%
\begin{equation*}
\mathcal{V}_{n,r}^{\left( s\right) }\left( x\right) =\underset{k=0}{\overset{%
n}{\sum }}B_{n+r,k+r}^{\left( r\right) }\left( \mathbf{a}_{s};\mathbf{e}+L%
\mathbf{a}_{s}\right) x^{k},\  \ s\geq 0.
\end{equation*}%
If $\mathcal{V}_{n,0}^{\left( s\right) }\left( x\right) $, then $\mathcal{V}%
_{n,r}^{\left( s\right) }\left( x\right) \in \mathbf{RZ}.$
\end{remark}

\begin{example}
For $a_{n}=\left( n-1\right) !,\ n\geq 1,$ we get $h\left( t\right) =-\ln
\left( 1-t\right) $ and%
\begin{align*}
A_{n}^{\left( 0\right) }\left( x\right) & =x\left( x+1\right) \cdots \left(
x+n-1\right) :=\left \langle x\right \rangle _{n}, \\
A_{n}^{\left( 1\right) }\left( x\right) & =A_{n}\left( \mathbf{B}_{\mathbf{x}%
}\right) =\left( \mathbf{B}_{\mathbf{x}}+n-1\right) _{n}=\mathcal{L}%
_{n}\left( x\right),
\end{align*}%
are in $\mathbf{RZ},$ then the polynomials%
\begin{equation*}
\mathcal{V}_{n,r}^{\left( 0\right) }\left( x\right) =\underset{k=0}{\overset{%
n}{\sum }}B_{n+r,k+r}^{\left( r\right) }\left( \mathbf{a}_{0};\mathbf{e}+L%
\mathbf{a}_{0}\right) x^{k},\  \  \mathcal{V}_{n,r}^{\left( 1\right) }\left(
x\right) =\underset{k=0}{\overset{n}{\sum }}B_{n+r,k+r}^{\left( r\right)
}\left( \mathbf{a}_{1};\mathbf{e}+L\mathbf{a}_{1}\right) x^{k},
\end{equation*}%
are also in $\mathbf{RZ}$, where $\mathbf{a}_{0}=(1!,2!,\ldots)$ and $%
\mathbf{a}_{1}=(\mathcal{L}_{1}\left( 1\right),\mathcal{L}_{2}\left(
1\right),\ldots)$.
\end{example}

\begin{theorem}
Let $r$ be a non-negative integer and $\left( f_{n}^{\left( r\right) }\left(
x\right) \right) $ be the sequence defined by%
\begin{equation*}
\underset{n\geq 0}{\sum }f_{n}^{\left( r\right) }\left( x\right) \dfrac{t^{n}%
}{n!}=F\left( t\right) \left( h\left( t\right) \right) ^{r}\exp \left(
xh\left( t\right) \right), \  \ h\left( t\right) =\underset{j\geq 1}{\sum }%
a_{j}\frac{t^{j}}{j!},
\end{equation*}%
for some power series $F.$ Then, for $r\leq n-1,$ if the polynomial $%
f_{n}^{\left( 0\right) }\left( x\right) $ is of degree $n$ and is in $%
\mathbf{RZ},$ then 
\begin{equation*}
f_{n}^{\left( r\right) }\left( x\right) =r!\overset{n}{\underset{k=r}{\sum }}%
\binom{n}{k}B_{k,r}\left( \mathbf{a}\right) f_{n-k}^{\left( 0\right) }\left(
x\right) \in \mathbf{RZ}.
\end{equation*}
\end{theorem}

\begin{proof}
From the definition of the sequence $\left( f_{n}^{\left( r\right) }\left(
x\right) \right) $ there holds $f_{n}^{\left( r\right) }\left( x\right) =%
\frac{d}{dx}f_{n}^{\left( r-1\right) }\left( x\right) .$ It follows that the
polynomial $f_{n}^{\left( r\right) }\left( x\right) $ is of degree $n-r.$
The proof can be deduced by induction on $r$ and by application of Rolle's
theorem.
\end{proof}

\begin{example}
Let $\left( f_{n}^{\left( r\right) }\left( x\right) \right) $ be the
sequence defined by%
\begin{equation*}
\underset{n\geq 0}{\sum }f_{n}^{\left( r\right) }\left( x\right) \dfrac{t^{n}%
}{n!}=\left( e^{t}-1\right) ^{r}\exp \left( x\left( e^{t}-1\right) \right).
\end{equation*}%
Then, the polynomial $f_{n}^{\left( r\right) }\left( x\right) =r!\overset{n}{%
\underset{k=r}{\sum }}\binom{n}{k} {n \brace k} \mathcal{B}_{n-k}\left(
x\right) $ is in $\mathbf{RZ}.$
\end{example}

\begin{example}
Let $\left( f_{n}^{\left( r\right) }\left( x\right) \right) $ be the
sequence defined by%
\begin{equation*}
\underset{n\geq 0}{\sum }f_{n}^{\left( r\right) }\left( x\right) \dfrac{t^{n}%
}{n!}=\left( \ln \left( 1+t\right) \right) ^{r}\exp \left( x\left( \ln
\left( 1+t\right) \right) \right).
\end{equation*}%
Then, the polynomial $f_{n}^{\left( r\right) }\left( x\right) =r!\overset{n}{%
\underset{k=r}{\sum }}\left( -1\right) ^{k-r}\binom{n}{k}{k \brack r }\left(
x\right) _{n-k}$ is in $\mathbf{RZ}.$
\end{example}

\begin{example}
Let $\left( f_{n}^{\left( r\right) }\left( x\right) \right) $ be the
sequence defined by%
\begin{equation*}
\underset{n\geq 0}{\sum }f_{n}^{\left( r\right) }\left( x\right) \dfrac{t^{n}%
}{n!}=\frac{1}{r!}\left( \frac{t}{1-t}\right) ^{r}\exp \left( \frac{xt}{1-t}%
\right).
\end{equation*}%
Then, the polynomial $f_{n}^{\left( r\right) }\left( x\right) =r!\overset{n}{%
\underset{k=r}{\sum }}\left( -1\right) ^{k-r}\binom{n}{k}L(k,r)\mathcal{L}%
_{n-k}\left( x\right) $ is in $\mathbf{RZ}$.
\end{example}

\end{document}